\documentclass[15pt,a4paper,reqno]{amsart}
\usepackage{amssymb}
\usepackage{amsmath,amsthm}
\usepackage{mathrsfs}
\usepackage{bbm}
\usepackage{color}
\def\bc{\begin{center}}
\def\ec{\end{center}}
\def\be{\begin{equation}}
\def\ee{\end{equation}}
\def\F{\mathcal F}
\def\N{\mathbb N}

\def\R{\mathbb R}
\def\Z{\mathbb Z}
\def\P{\mathcal P}

\def\ep{\epsilon}
\def\E{\mathcal E}

\def\K{\mathcal{K}}
\def\1{{\mathbbm{1}}}
\newtheorem{lem}{Lemma}[section]

\newtheorem{dfn}[lem]{Definition}

\newtheorem{thm}[lem]{Theorem}

\newtheorem{cor}[lem]{Corollary}
\theoremstyle{remark}
\newtheorem*{rem}{Remark}
\numberwithin{equation}{section}

\begin{document}
\title[]
{Mahler's question for intrinsic Diophantine approximation on triadic Cantor set: the divergence theory
}

\author{Bo Tan}\author{Baowei WANG}\author{Jun Wu}
\address{School of Mathematics and Statistics, Huazhong University of Science and Technology, 430074 Wuhan, China}
\email{bo.tan@hust.edu.cn}
\email{bwei\_wang@hust.edu.cn}\email{jun.wu@mail.hust.edu.cn}



 \maketitle

\begin{abstract}
In this paper, we consider the intrinsic Diophantine approximation on the triadic Cantor set $\K$, i.e. approximating the points in $\K$ by rational numbers inside $\K$, a question posed by K. Mahler. By using another height function of a rational number in $\K$, i.e. the denominator obtained from its periodic 3-adic expansion, a complete metric theory for this variant intrinsic Diophantine approximation is presented which yields the divergence theory of Mahler's original question.
\end{abstract}

\section{Introduction}
Metric Diophantine approximation studies the distribution of rational numbers in the quantitative sense. To some extent, the distribution of rational numbers in $\mathbb{R}^d$ has been well understood, since the pioneer work of Khintchine (Lebesgue measure theory) \cite{Khint}, Jarnik (Hausdorff measure theory) \cite{Jarn1} until the recent outstanding contributions by Beresnevich \& Velani \cite{BV06}, Beresnevich, Dichinson \& Velani \cite{BDV06} and Koukoulopoulos \& Maynard \cite{KouM}.

 As a further study, instead of studying the distribution of rational numbers in the whole space $\mathbb{R}^d$, one can also study it restricted to some subsets of $\mathbb{R}^d$, which leads to the study of Diophantine approximation on {\em manifolds} and on {\em fractals}.

 For the study on Diophantine approximation on manifolds, one is referred to the works \cite{Be,B2,KleiM,Sp} and references therein which is originated from Diophantine approximation on Veronese curves posed by K. Mahler \cite{Mah1} in 1932.

In this paper, we focus on the latter scheme. The study of Diophantine approximation on fractals was raised also by K. Mahler in 1984. Throughout, let $\K$ be the triadic Cantor set, $\mu$ the standard Cantor measure supported on $\K$, and $\gamma=\log_3 2$ the Hausdorff dimension of $\K$. In Mahler's paper entitled `some suggestions for further research' \cite{Mah}, Mahler wrote the following moving statement: ``At the age of 80, I cannot expect to do much
more mathematics. I may however state a number of questions where perhaps further research
might lead to interesting results". One of these questions was regarding intrinsic and extrinsic
approximation on the Cantor set $\K$. In Mahler's words, how close can irrational elements of
 $\K$ be approximated by rational numbers
\begin{itemize}
  \item inside the Cantor set $\K$ (intrinsic), or
\item outside the Cantor set $\K$ (extrinsic)?'
\end{itemize}

Over the last twenty years, much attention has been focused on Diophantine approximation on fractals. Before recall the related progress, let's give some notation at first. \begin{itemize}\item A point $x\in \mathbb{R}$ is called {\em $\psi$-well approximable} if there exist infinitely many rational numbers $p/q$ such that $$
|x-p/q|<{\psi(q)},
$$ where $\psi:\mathbb{R}^+\to \mathbb{R}^+$ is a positive function, and specially call $x$ being {\em $v$-well approximable} if $\psi(q)=q^{\-v}$.
Call $x\in \mathbb{R}$ {\em very well approximable} if it is $v$-well approximable for some $v>2$.

\item A point $x\in \mathbb{R}$ is called having {\em exact approximation order} (or irrationality measure) $v$, if it is $v$-well approximable but not $(v+\epsilon)$-well approximable for any $\epsilon>0$.

\item A point $x\in \mathbb{R}$ is called {\em badly approximable} if there exists a constant $c>0$ such that $$
|x-p/q|>\frac{c}{q^2},  \ {\text{for all}}\ p/q.
$$

\item
Let $\mathcal{C}$ be a subset of $\mathbb{R}$. A point $x\in \mathcal{C}$ is called {\em intrinsically $\psi$-well approximable} if there exist infinitely many rationals $p/q$ in $\mathcal{C}$, such that $$
|x-p/q|<{\psi(q)}.
$$
\end{itemize}

Mahler's question was first studied by B. Weiss \cite{Weiss}, who showed that almost no points in $\K$ are very well approximable
with respect to the Cantor measure $\mu$.  This is also termed as that $\mu$ is  extremal. The extremality of a measure supported on fractals
was extended to the measure with `friendly' properties, see for examples, 
\cite{KleinLW}, 
\cite{PollV} and
\cite{Klein}. There are also rich results on the size of badly approximable points on fractals, for example, the measure in \cite{Einsiedler}  and \cite{SimmB}, the dimension in \cite{KleinW} and \cite{KristTV}, the winning properties in \cite{Fishman}, \cite{Broderich} and \cite{Broderich2}.

However, Mahler's question is still not fully solved. 
As a first step towards how well the points in $\K$ can be approximated by rational numbers, Levesley, Salp \& Velani \cite{Levesley} and Bugeaud \cite{Bug} constructed explicit points in $\K$ with exact approximation order $v$ for any $v>2$. To simulate the possible distributions of rational numbers near $\K$, Bugeaud \& Durand \cite{BugD} presented a model by allowing random translations of $\K$ and established the dimensional theory. Their dimensional result may give the potential right answer to Mahler's question, but unfortunately Mahler's question is still open.

To obtain more quantitative information on the distribution of rational numbers inside $\K$,  Broderich, Fishman \& Reich \cite{Broderich3} presented a Dirichlet-type result.\begin{thm}[Broderich, Fishman \& Reich, \cite{Broderich3}] For any $x\in \K$ and any $Q>1$, there exists $p/q\in \K$ with $1\le q\le Q$ such that $$
|x-p/q|<\frac{1}{q(\log_3Q)^{1/\gamma}},
$$  where $\gamma=\log_32$ the Hausdorff dimension of $\K$.\end{thm} This Dirichlet-type result was also shown to be optimal by Fishman \& Simonns \cite{FishS} and Fishman, Merrill \& Simmons \cite{FishMS}.

By restricting on a special class of rational numbers in $\K$, Levelsey, Salp \& Velani established the following Khintchine/Jarn\'{i}k-type result. Let $$
\mathcal{A}=\{3^n: n\ge 1\}.
$$ Define $$
\mathcal{W}_{\mathcal{A}, \K}(\psi)=\Big\{x\in \K: |x-p/q|<\psi(q), \ {\text{i.m.}},\ p/q\in \K, \ (p,q)\in \mathbb{Z}\times \mathcal{A}\Big\}
$$ where {\em i.m.} denotes {\em infinitely many} for brevity. \begin{thm}[Levesley, Salp \& Velani, \cite{Levesley}]
  Let $\psi:\mathbb{R}^+\to \mathbb{R}^+$. Let $f$ be a dimension function such that $f(r)/r^{\gamma}$ is non-decreasing. Then\begin{align*}
  \mathcal{H}^f\Big(\mathcal{W}_{\mathcal{A}, \K}(\psi)\Big)=\left\{
                                                       \begin{array}{ll}
                                                         \mathcal{H}^f(\K), & \hbox{if $ \sum_{n\ge 1}3^{n\gamma}f(\psi(3^n))=\infty$;}\medskip \\
                                                         0, & \hbox{otherwise,}
                                                       \end{array}
                                                     \right.
  \end{align*} where $\mathcal{H}^f$ denotes $f$-Hausdorff measure.
\end{thm}

The setting of Levesley, Salp \& Velani \cite{Levesley} was generalized to conformal iterated functions systems by Baker \cite{Baker1} and Demi \& B\'{a}r\'{a}ny \cite{Allen1}. So, all of their results can give a divergence theory for Mahler's first question. But as will be seen below, their applications to Mahler's first question is not optimal. The main reason is that the rational numbers $p/q\in \K$ with $q\in \mathcal{A}$ are just those with terminating 3-adic expansions, so only a subfamily of rational points in $\K$. Therefore to attack Mahler's first question, one has to consider the distributions of {\em all} rational numbers in $\K$.

A first simple observation is that in $3$-adic expansion, every rational number $p/q\in \K$ can be expressed as an ultimately periodic series, denoted by \begin{equation}\label{f3}
p/q=[\epsilon_1,\cdots,\epsilon_{\ell}, (w_1,\cdots, w_m)^{\infty}], \ {\text{with}}\ \epsilon_i, w_j\in \{0,2\}.
\end{equation}
By a calculation, the number in the right side of (\ref{f3}) can be written as a fraction \begin{equation}\label{f4}
\frac{p}{q}=\frac{p^*}{3^{\ell}(3^m-1)}.
\end{equation}

At the current stage, the biggest obstacle in achieving a complete answer of Mahler's first question is that the rational in the right side of (\ref{f4}) is not reduced. Its reduced form depends on the length of the period of the expansion of $p/q$ and the divisors of $3^m-1$ where a general characterization is far out of reach.

So instead of the denominator, Fishman \& Simonns \cite{FishS} used $3^{\ell}(3^m-1)$ given in (\ref{f4}), denoted by $q_{{\text{int}}}$, as the height of $p/q$, and call it the {\em intrinsic denominator} of $p/q$. This leads to a variant form of Mahler's first question: consider the size of the set $$
\mathcal{W}_{{\text{int}}, \K}(\psi):=\Big\{x\in \K: |x-p/q|<\psi(q_{{\text{int}}}), \ {\text{i.m.}}\ p/q\in \K\Big\}.
$$
It was shown that \begin{thm}[Fishman \& Simonns, \cite{FishS}] Let $\psi: \R^+\to \R^+$ be a non-increasing positive function. Then
\begin{align*}
  \mu\Big(\mathcal{W}_{{\text{int}}, \K}(\psi)\Big)=\left\{
                                                      \begin{array}{ll}
                                                        0, & \hbox{if $\sum_{n\ge 1}n\big(3^n\psi(3^n)\big)^{\gamma}<\infty$;}\bigskip \\
                                                        1, & \hbox{if $\sum_{n\ge 1}\frac{n(3^n\psi(3^n))^{\gamma}}{-\log (3^n\psi(3^n))}=\infty$.}\medskip
                                                      \end{array}
                                                    \right.
\end{align*}
\end{thm}
The above two series deciding the size of $\mathcal{W}_{{\text{int}}, \K}(\psi)$ are inconsistent with each other, so a complete dichotomy law is desired. We solve this by showing that
 \begin{thm}\label{t1} Let $\psi: \R^+\to \R^+$ be a non-increasing positive function. Then
\begin{align*}
  \mu\Big(\mathcal{W}_{{\text{int}}, \K}(\psi)\Big)=\left\{
                                                      \begin{array}{ll}
                                                        0, & \hbox{if $\sum_{n\ge 1}n\big(3^n\psi(3^n)\big)^{\gamma}<\infty$;}\bigskip \\
                                                        1, & \hbox{if $\sum_{n\ge 1}n\big(3^n\psi(3^n)\big)^{\gamma}=\infty$.}\medskip
                                                      \end{array}
                                                    \right.
\end{align*} Consequently, by letting $f$ be a dimension function such that $f(r)/r^{\gamma}$ is non-decreasing, then\begin{align*}
  \mathcal{H}^f\Big(\mathcal{W}_{{\text{int}}, \K}(\psi)\Big)=\left\{
                                                       \begin{array}{ll}
                                                         \mathcal{H}^f(\K), & \hbox{if $ \sum_{n\ge 1}n3^{n\gamma}f(\psi(3^n))=\infty$;}\medskip \\
                                                         0, & \hbox{otherwise.}
                                                       \end{array}
                                                     \right.
  \end{align*}
\end{thm}

It is always believed that the divergence theory is harder than the convergence one. Compared with the method used in \cite{FishS}, for the convergence part, we also use the convergence part of the Bore-Cantelli lemma (though the argument is little different); for the divergence part, a counting lemma about ``good spaced" rational numbers in $\K$ is essential to our argument.

Since $\psi$ is non-increasing and $q\le q_{\text{int}}$ for a rational number $p/q$, one has that $$
\mathcal{W}_{{\text{int}}, \K}(\psi)\subset \Big\{x\in \K: |x-p/q|<\psi(q), \ {\text{i.m.}}\ p/q\in \K\Big\}.
$$ Thus we have the divergence theory for Mahler's first question. \begin{cor}\label{c1}
Let $\psi$ be a non-increasing positive function. Then $$
\mu\left(\Big\{x\in \K: |x-p/q|<\psi(q), \ {\text{i.m.}}\ p/q\in \K\Big\}\right)=1,
$$ i.e., almost all points in $\K$ are intrinsically $\psi$-well approximable when $$\sum_{n\ge 1}n\big(3^n\psi(3^n)\big)^{\gamma}=\infty.$$
\end{cor}

There are also some other interesting results related to Mahler's questions. For example, in \cite{FishSi}, Fishman \& Simmons considered the extrinsic Diophantine approximation, i.e. Mahler's second question; in \cite{Kris}, Kristensen considered the approximation of points in $\K$ by algebraic numbers; in \cite{Allen}, Allen, Chow \& Yu considered the approximating the points in $\K$ by dyadic rational numbers. There are also two excellent works without restricting the rationals inside the Cantor set: in \cite{Khalil} Khalil \& Luethi established the Khintchine's theorem on some fractals, for example, generated by similarities and with sufficiently large Hausdorff dimension; in \cite{Han} Han established the Khintchine's theorem for some fractals with large $l_1$-dimension, for example middle $b$-adic Cantor set with $b$ large.
However, both of them cannot be applied to the case of triadic Cantor set.

 The paper is organized as follows. In the next section, we fix some notation and give some preliminaries. The convergence theory is proved in Section 3 which follows a counting lemma in Section 4.  The divergence theory is proved in Section 5 and in the last section, we give some words on the potential convergence theory for Mahler's first question.

\section{preliminaries}

\subsection{Notation}\

Throughout the paper, the expansion of a real number is its $3$-adic expansion. Since all the points we are concerning come from the Cantor set $\K$, so all the digits related to the $3$-adic expansion, for examples $\epsilon_i, w_j$ below, lie in $\E=\{0,2\}$.

For a point $x\in \K$, write its $3$-adic expansion as $$
x=\frac{\epsilon_1}{3}+\frac{\ep_2}{3^2}+\cdots+\frac{\ep_n}{3^n}+\cdots.
$$ Call $\{\ep_i\}_{i\ge 1}$  the {\em digit sequence} of $x$ and write the above series simply as $$
x=[\ep_1, \ep_2,\cdots, \ep_n,\cdots],
$$ and write $$
[\epsilon_1,\cdots, \epsilon_{\ell}]=\frac{\epsilon_1}{3}+\frac{\epsilon_2}{3^2}+\cdots+\frac{\epsilon_{\ell}}{3^{\ell}},
$$ throughout this paper.
For a rational number $p/q$ with ultimately periodic expansion, saying \begin{align}\label{f5}
\frac{p}{q}=\frac{\epsilon_1}{3}+\cdots+\frac{\epsilon_{\ell}}{3^{\ell}}+
\frac{w_1}{3^{\ell+1}}+\cdots+\frac{w_m}{3^{\ell+m}}+\frac{w_1}{3^{\ell+m+1}}+\cdots+\frac{w_m}{3^{\ell+2m}}+\cdots,
\end{align} we write it as \begin{equation*}
p/q=[\epsilon_1,\cdots, \epsilon_{\ell}, (w_1,\cdots, w_m)^{\infty}].
\end{equation*}

Roughly speaking, there are multiple choices of $\ell, m$ in the expansion of $p/q$, so to make the {\em intrinsic denominator} well defined, we give the following notation. \begin{dfn}
  Let $p/q$ be a rational number in $\K$ with the digit sequence $\{\ep_i\}_{i\ge 1}$. Let $\ell\ge 0$ be the smallest integer such that $\{\ep_i\}_{i>\ell}$ is purely periodic. Then let $m$ be the smallest length of the period of $\{\ep_i\}_{i>\ell}$. Define $$
\mathcal{L}(p/q)=\ell, \ \ \ \mathcal{P}(p/q)=m.
$$ Call them the {\em prelength} and {\em period} of $p/q$ respectively.
\end{dfn}

Let $p/q$ be a rational in $\K$ with prelength $\mathcal{L}(p/q)=\ell$ and period $\mathcal{P}(p/q)=m$. So it can be written as (\ref{f5}).
Then by a calculation of the number in the right side of (\ref{f5}), one gets \begin{equation*}
\frac{p}{q}=\frac{3^{\ell+m}[\epsilon_1,\cdots,\epsilon_{\ell}, w_1,\cdots, w_m]-3^{\ell}[\epsilon_1,\cdots,\epsilon_{\ell}]}{3^{\ell}(3^m-1)}.
\end{equation*}

\begin{dfn}\label{d1}
  Denote $$
p_{{\text{int}}}=3^{\ell+m}[\epsilon_1,\cdots,\epsilon_{\ell}, w_1,\cdots, w_m]-3^{\ell}[\epsilon_1,\cdots,\epsilon_{\ell}], \ \ q_{{\text{int}}}=3^{\ell}(3^m-1).
$$ Call them the {\em intrinsic numerator/denominator} of $p/q$ respectively.
\end{dfn}

For each block of digits $(\ep_1,\cdots,\ep_n)\in \{0,2\}^n$, call $$
I_n(\ep_1,\cdots, \ep_n)=\Big\{x\in \K: {\text{the 3-adic expansion of $x$ begins with}}\ (\ep_1,\cdots,\ep_n)\Big\}
$$ a cylinder of order $n$.
We collect some well known facts for later use. 
\begin{lem}\label{l1} \

\begin{itemize}\item Different cylinders of order $n$ are separated by a gap at least $3^{-n}$. So, for any $$
x=[\ep_1,\cdots, \ep_n,\cdots]\in \K, \ y=[w_1,\cdots, w_n,\cdots]\in \K
$$ if $\ep_n\ne w_n$ for some $n\ge 1$, then $$
|x-y|\ge 3^{-n}.
$$

\item For any cylinder $I_n(\ep_1,\cdots, \ep_n)$ of order $n$, $$
\mu(I_n(\ep_1,\cdots, \ep_n))=3^{-n\gamma}=2^{-n}.
$$

\item The measure $\mu$ is $\gamma$-Ahlfors regular in the sense that there exists constants $c>0$ and $r_0>0$ such that for all $x\in \K$ and $r<r_0$, $$
c\cdot r^{\gamma}\le \mu(B(x,r))\le c^{-1}\cdot r^{\gamma}.
$$ 
\end{itemize}
\end{lem}

\subsection{Chung-Erd\"{o}s inequality}

To get the measure of a limsup set in a probability space, the following results are widely used. Let $(\Omega, \mathcal{B}, \nu)$ be a probability space and $\{E_n\}_{n\ge 1}$ be a sequence of measurable sets. Define $$
E=\limsup_{n\to \infty}E_n.
$$

\begin{lem}[Borel-Cantelli lemma] \begin{align*}
  \nu(E)=\left\{
         \begin{array}{ll}
           0, & \hbox{if $\sum_{n\ge 1}\nu(E_n)<\infty$;} \\
           1, & \hbox{if $\sum_{n\ge 1}\nu(E_n)=\infty$ \ and $\{E_n\}_{n\ge 1}$ are pairwise independent.}
         \end{array}
       \right.
\end{align*}
\end{lem}

In applications, the measurable sets $\{E_n\}_{n\ge 1}$ are always not pairwise independent, so for the divergence part, one uses the following lemma instead.
\begin{lem}[Chung-Erd\"{o}s inequality, \cite{Chung}] If $\sum_{n\ge 1}\nu(E_n)=\infty$, then $$
\nu(E)\ge \limsup_{n\to \infty}\frac{(\sum_{1\le n\le N}\nu(E_n))^2}{\sum_{1\le i\ne j\le N}\nu(E_i\cap E_j)}.
$$
\end{lem}

In many cases, Chung-Erd\"{o}s inequality enables one to conclude the positiveness of $m(E)$, so to get a full measure result for $E$, one can apply Chung-Erd\"{o}s inequality locally, i.e. apply it to the set $E\cap B$ for any ball $B\subset \Omega$. Then one arrives at the full measure of $E$ in the light of the following result.

A measure $\nu$ in a metric space $(\Omega, |\cdot|)$ is called doubling, if there exists $r_0>0, c_0>0$ such that for any $x\in \Omega$ and $r<r_0$, $$
\nu(B(x,2r))\le c_0\cdot \nu(B(x,r)).
$$
\begin{lem}\label{l2}
Let $(\Omega, |\cdot|)$ be a metric space and let $\nu$ be a doubling
probability measure on $\Omega$ such that any open set is measurable. Let $E$ be a Borel subset of $\Omega$.
Assume that there are constants $r_0, c >0$ such that for any ball $B$ of radius $r(B) < r_0$
and centered in $\Omega$ we have
$$\nu(E \cap B) \ge c\cdot \nu(B).$$
Then $\nu(E) = 1.$
\end{lem}

\subsection{Mass transference principle}

The outstanding mass transference principle established by Beresnevich \& Velani enables one to transfer a full (Lebesgue) measure statement for a limsup set to the full Hausdorff measure statement. So, once a full (Lebesgue) measure statement is presented, the Hausdorff measure statement follows directly.

Let $\Omega$ be a locally compact metric space, and $g$ a doubling dimension function (
$\exists \ \lambda\ge 1$ such that $g(2r)\le \lambda g(r)$ for all $r>0$ small enough).
Suppose that there exist $0<c_1\le c_2<\infty$ and $r_o>0$ such that for all $x\in \Omega$ and $0<r<r_o$,
\begin{equation*}
c_1g(r)\le \mathcal{H}^{g}(B(x,r))\le c_2 g(r).
\end{equation*}
Let $f$ be a dimension function and write $B^f(x,r)$ for the ball $B(x, g^{-1} (f(r)))$.

\begin{thm}[Beresnevich \& Velani \cite{BV06}]
Let $\Omega$ be a locally compact metric space, $f$ a dimension function and $g$ a doubling dimension function. Assume that $\{B_i\}_{i\in \N}$ is a sequence of balls in $\Omega$ with radii tending to $0$, and that $\frac{f(r)}{g(r)}$ increases as $r\to 0_+$.
If, for any ball $B$ in $\Omega$,
$$\mathcal{H}^{g}\Big(B \cap \limsup_{i\to \infty}B_i^f\Big)= \mathcal{H}^{g}(B);$$
then, for any ball $B$ in $\Omega$,
$$\mathcal{H}^f\Big(B \cap \limsup_{i\to \infty}B_i^g\Big)= \mathcal{H}^f(B).$$
\end{thm}

So, the Hausdorff measure result for the set $\mathcal{W}_{{\text{int}}, \K}(\psi)$ in Theorem \ref{t1} is just a consequence of its $\mu$-measure result.


\section{Convergence part}

Recall the set we are considering: $$
\mathcal{W}_{{\text{int}}, \K}(\psi):=\Big\{x\in \K: |x-p/q|<\psi(q_{{\text{int}}}), \ {\text{i.m.}}\ p/q\in \K\Big\}.
$$

We classify the rational numbers $p/q$ according to their prelength and period: for each $n\ge 1$, define $$
\mathcal{T}_n=\Big\{
p/q\in \K: \mathcal{L}(p/q)=\ell, \mathcal{P}(p/q)=m, \ \ell+m=n
\Big\}.
$$ Note that for each $p/q\in \mathcal{T}_n$, $$
3^{n-1}\le q_{\text{int}}=3^{\ell}(3^m-1)\le 3^n,
$$ and moreover \begin{align*}
\sharp \mathcal{T}_n&=\sum_{\ell+m=n}\sharp\Big\{p/q=[\ep_1,\cdots, \ep_{\ell}, (w_1,\cdots, w_m)^{\infty}]: \mathcal{L}(p/q)=\ell, \mathcal{P}(p/q)=m\Big\}\\
&\le \sum_{\ell+m=n}2^{\ell+m}=n2^n.
\end{align*}

Since $\psi$ is non-increasing and $q_{{\text{int}}}\ge 3^{n-1}$ for $p/q\in \Gamma_n$, one has \begin{align*}
  \mathcal{W}_{{\text{int}}, \K}(\psi)\subset \Big\{x\in \K: |x-p/q|<\psi(3^{n-1}),\ p/q\in \mathcal{T}_n, \ {\text{i.m.}}\ n\in \N\Big\}.
\end{align*}
Thus by the convergence part of the Borel-Cantelli lemma, we will have $\mu(\mathcal{W}_{{\text{int}}, \K}(\psi))=0$ if\begin{align*}
\sum_{n=1}^{\infty}\sum_{p/q\in \mathcal{T}_n}\mu\Big(B(p/q, \psi(3^{n-1}))\Big)\le \sum_{n=1}^{\infty}n\cdot 2^n\Big(\psi(3^{n-1})\Big)^{\gamma}<\infty.
\end{align*} This is true since by the monotonicity of $\psi$, $$
\sum_{n=1}^{\infty}n\cdot 2^n\Big(\psi(3^{n-1})\Big)^{\gamma}<\infty\Longleftrightarrow \sum_{n=1}^{\infty}n\Big(3^n\psi(3^{n})\Big)^{\gamma}<\infty.
$$ This proves the convergence part of Theorem \ref{t1}.

\section{A Counting lemma}
In this section, we will present a large collection of well separated rational numbers in $\K$ with restricted denominators which is crucial for the divergence theory of  $\mathcal{W}_{{\text{int}}, \K}(\psi)$.

\subsection{A subset of $\mathcal{W}_{{\text{int}}, \K}(\psi)$}\

For a rational number $p/q\in \K$ with prelength $\mathcal{L}(p/q)=\ell$ and period $\mathcal{P}(p/q)=m$, one has that $$
q_{\text{int}}=3^{\ell}(3^m-1).
$$ Formally, $p/q$ can be written as other ultimately periodic series, saying, $$
p/q=[\ep_1,\cdots, \ep_{\ell'}, (w_1,\cdots, w_{m'})^{\infty}].
$$ By the minimality of $\ell, m$, one knows that $$
q_{\text{int}}\le 3^{\ell'+m'}.
$$ So by monotonicity of $\psi$, we have the following subset of $\mathcal{W}_{{\text{int}}, \K}(\psi)$:\begin{align*}
 \widetilde{\mathcal{W}}:=\Big\{x\in \K: \big|x-[\ep_1,&\cdots, \ep_{\ell}, (w_1,\cdots, w_{m})^{\infty}]\big|<\psi(3^{\ell+m}), \\
 &\ep_i, w_j\in \{0,2\},\  1\le i\le \ell, 1\le j\le m,\ {\text{i.m.}}\ (\ell, m)\in \Z_{\ge 0}^2\Big\}.
\end{align*}
In other words, without the restrictions that $\ell$ and $m$ must be the prelength and period of a rational number respectively, one still get a subset of subset (i.e. $\widetilde{\mathcal{W}}$) of $\mathcal{W}_{{\text{int}}, \K}(\psi)$.
In the following, we focus on the $\mu$-measure of $\widetilde{\mathcal{W}}$.

\subsection{Rational numbers in $\K$}\

We give a mechanism to generate all the rational numbers in $\K$. For each block of digits $(\ep_1,\cdots, \ep_n)$ with $n\ge 1$, formally, it can generate $(n+2)$ rational numbers in $\K$: \begin{itemize}
\item the endpoints of the 3-adic cylinders:
$$
[\epsilon_1,\cdots, \epsilon_n, 0^{\infty}], \ \ \ [\epsilon_1,\cdots, \epsilon_n, 2^{\infty}];
$$

\item ultimately periodic ones: $$
[\epsilon_1,\cdots,\epsilon_{\ell}, (\epsilon_{\ell+1},\cdots,\epsilon_n)^{\infty}],\ 0\le \ell<n.
$$
\end{itemize}
This gives all rational numbers by union over $(\ep_1,\cdots, \ep_n)$ and $n\ge 1$, moreover a classification of them.

To study the measure of $\widetilde{\mathcal{W}}$, we hope that the rational numbers $$[\epsilon_1,\cdots,\epsilon_{\ell}, (w_{1},\cdots,w_m)^{\infty}]$$
are sufficiently well distributed. But for the rational numbers given above,
some of them may be the same, and some of them may be sufficiently close. So, we need seek a subfamily of well spaced rational numbers and hope that this subfamily can contain as many as possible rational numbers.

Denote $\mathcal{C}_t(\ep_1,\cdots,\ep_n)$ the number of different subwords of length $t$ appearing in $(\ep_1,\cdots,\ep_n)$, i.e. its complexity. Let $c_1=1/4, c_2=1/16$. Let $k_n=\lfloor\log_2 n\rfloor$ for each $n\ge 1$. Define
$$\mathcal{F}_n:=\Big\{(\ep_1,\cdots,\ep_n)\in \{0,2\}^n: \mathcal{C}_{k_n}(\ep_1,\cdots,\ep_n)\ge c_1 n\Big\}.$$
\begin{lem}[A counting lemma]\label{l3}
 $$
\sharp \F_n\ge c_2\cdot 2^n.
$$
\end{lem}
\begin{proof}
We prove this by using a probability model. Write $k=k_n$. For a finite word $(x_1 x_2\cdots x_n)\in \{0,2\}^n$ and  $\xi\in \{0,2\}^k$, let $$
|x_1x_2\cdots x_n|_{\xi}=\sharp\{i: 0\le i\le n-k, x_{i+1}x_{i+2}\cdots x_{i+k}=\xi\}
$$ and $$
\Sigma_{k}^n(x_1x_2\cdots x_n)=\sum_{\xi\in \{0,2\}^k}|x_1x_2\cdots x_n|^2_{\xi}.
$$

Let $X_1,X_2,\cdots$ be i.i.d. random variables with the distribution $P(X_i=0)=P(X_i=2)=1/2$. Write $$
\Sigma_{k}^n=\sum_{\xi\in \{0,2\}^k}|X_1X_2\cdots X_n|^2_{\xi}.
$$

At first, we estimate the expectation of the random variable $\Sigma_{k}^n$:\begin{align*}
\mathbb{E}(\Sigma_k^n)&=\sum_{\xi\in \{0,2\}^k}\mathbb{E}\Big[|X_1X_2\cdots X_n|^2_{\xi}\Big]\\
&=\sum_{\xi\in \{0,2\}^k}\mathbb{E}\left[\left(\sum_{i=0}^{n-k}\1 _{\xi}(X_{i+1}X_{i+2}\cdots X_{i+k})\right)^2\right]\\
&=\sum_{\xi\in \{0,2\}^k}\mathbb{E}\left[\sum_{i,j=0}^{n-k}\1_{\xi}(X_{i+1}X_{i+2}\cdots X_{i+k})\cdot \1_{\xi}(X_{j+1}X_{j+2}\cdots X_{j+k})\right]\\
&=\sum_{i,j=0}^{n-k}\mathbb{E}\left[\sum_{\xi\in \{0,2\}^k}\1_{\xi}(X_{i+1}X_{i+2}\cdots X_{i+k})\cdot \1_{\xi}(X_{j+1}X_{j+2}\cdots X_{j+k})\right]\\
&=\sum_{i,j=0}^{n-k}\mathbb{E}\left[\1_{X_{i+1}X_{i+2}\cdots X_{i+k}=X_{j+1}X_{j+2}\cdots X_{j+k}}\right].
\end{align*}
Since \begin{align*}
  P\Big(X_{i+1}X_{i+2}\cdots X_{i+k}=X_{j+1}X_{j+2}\cdots X_{j+k}\Big)=\left\{
                                                                 \begin{array}{ll}
                                                                   1, & \hbox{when $i=j$;} \\
                                                                   2^{-k}, & \hbox{when $i\ne j$,}
                                                                 \end{array}
                                                               \right.
\end{align*} we have that \begin{align*}
\mathbb{E}(\Sigma_k^n)&=\sum_{i=0}^{n-k}1+\sum_{i,j=0, i\ne j}^{n-k}2^{-k}=(n-k+1)+(n-k)(n-k+1)2^{-k}\\
&\le n+n^2\cdot 2n^{-1}\le 3n.
\end{align*}

On the other hand, let $B$ be the event that $$
\sharp\Big\{\xi\in \{0,2\}^k: |X_1X_2\cdots X_n|_{\xi}\ge 1\Big\}\le n/4:=\ell,
$$ i.e. the event that the number of different subwords of $X_1X_2\cdots X_n$ is less than $n/4$. We will show that the probability of $B$ is strictly less than 1. This would yield our desired result, since there is a positive probability that the number of different subwords of $X_1X_2\cdots X_n$ is larger than $n/4$.

Notice that \begin{align*}
\mathbb{E}(\Sigma_k^n)&\ge \mathbb{E}\left[\sum_{\xi\in \{0,2\}^k}|X_1X_2\cdots X_n|^2_{\xi}\ \bigg|B \right]P(B)\\
&\ge P(B)\cdot\min\{z_1^2+\cdots+z_{\ell}^2: z_1+\cdots+z_{\ell}=n-k+1\}\\
&=P(B)\cdot\frac{(n-k+1)^2}{\ell}\ge P(B)\cdot\frac{4(n-k+1)^2}{n}\\
&\ge P(B)\cdot\frac{4\cdot 4/5 \cdot n^2}{n}=P(B)\cdot\frac{16n}{5}.
\end{align*} Thus, it follows that
$$
P(B)\le \frac{15}{16}<1.
$$ Thus, the
number of subwords of length $k$ in $X_1 X_2 \cdots X_n$ is larger than $n/4$ with a
positive probability.
\end{proof}

%
%
%
%
%
%
%
%

Choose a sub-collection of $\F_n$ with cardinality $c_2\cdot 2^n$ and still denote it by $\F_n$. For each $(w_1,\cdots,w_n)\in \F_n$, let $$
\mathcal{G}(w_1,\cdots,w_n)=\Big\{(w_{\ell_i+1},\cdots,w_{\ell_i+k_n}):1\le i\le c_1n\Big\}
$$ be a sub-collection of different sub-words of $(w_1,\cdots,w_n)$ of length $k_n$.

We will use Erd\"{o}s-Chung inequality locally. So, fix a cylinder $B=I_{t}(\ep_1,\cdots, \ep_t)$. For each $(w_1,\cdots,w_n)\in \F_n$,
 denote by $\mathcal{P}_B(w_1,\cdots,w_n)$ a collection of rational points in $\K$ as
\begin{equation*}
\mathcal{P}_B(w_1,\cdots,w_n)=\Big\{[\ep_1,\cdots,\ep_t, w_1,\cdots, w_{\ell_i}, (w_{\ell_i+1},\cdots,w_n)^{\infty}]: 1\le i\le c_1 n\Big\}.
\end{equation*}

By Lemma \ref{l1}, we have the following observations: \begin{enumerate}
\item for any two different words $(w_1,\cdots,w_n)$ and $(w_1',\cdots, w_n')$,
$$
{\text{dist}}\Big(\P_B(w_1,\cdots,w_n), \P_B(w'_1,\cdots,w_n')\Big)\ge 3^{-n-t};
$$

\item for any two different points $p_i/q_i$ and $p_{i'}/q_{i'}$ in $\P_B(w_1,\cdots, w_n)$, $$
|p_i/q_i-p_{i'}/q_{i'}|\ge \frac{1}{3^{t+n+k_n}}.
$$

\item for any $p/q\in \P_B(w_1,\cdots,w_n)$, $$
q_{\text{int}}\le 3^{t+n}.
$$

\end{enumerate}
\begin{rem}
By the definition of the complexity of a word, for $k<k'$, one observes that $$
\mathcal{C}_{k'}(w_1,\cdots, w_n)\ge \mathcal{C}_k(w_1,\cdots, w_n)-(k'-k).
$$ Thus Lemma \ref{l3} is still true if we replace $k_n=\lfloor \log_2 n\rfloor$ by any $k'$ (small compared with $n$), say $k<k'\le n/8$. But the choice $k_n=\lfloor \log_2 n\rfloor$ is essential in the argument for the divergence theory (see (ii) in estimation the correlation $\mu(E_m\cap E_n)$ below).
\end{rem}

\section{Divergence part}

Recall that $B=I_t(\ep_1,\cdots,\ep_t)$ be a fixed ball. We assume that \begin{equation}\label{f1}
\psi(3^{n+t})\le \frac{1}{4}\cdot \frac{1}{3^{n+t+k_n}}, \ {\text{for all}}\ n\gg 1.
\end{equation} Otherwise, let $$
\psi'(3^{n+t})=\min\Big\{\psi(3^{n+t}), \ \frac{1}{4}\cdot\frac{1}{3^{n+t+k_n}}\Big\}.
$$ Note that for those $n$ with $\psi(3^{n+t})\ge \frac{1}{4\cdot 3^{n+t+k_n}}$, one has
$$
n \Big(3^{n+t}\psi'(3^{n+t})\Big)^{\gamma}=n \Big(3^{n+t}\cdot \frac{1}{4}\cdot \frac{1}{3^{n+t+k_n}}\Big)^{\gamma}=n\Big(\frac{1}{4}\cdot 3^{-k_n}\Big)^{\gamma}=4^{-\gamma},
$$ where for the last equality, we used that $3^{k_n\gamma}=n$. Thus it follows that $$
\sum_{n\ge 1}n \Big(3^{n+t}\psi(3^{n+t})\Big)^{\gamma}=\infty \Longrightarrow \sum_{n\ge 1}n \Big(3^{n+t}\psi'(3^{n+t})\Big)^{\gamma}=\infty.
$$

For each $n\ge 1$, define $$
E_n=\bigcup_{(w_1,\cdots,w_n)\in \F_n}\ \bigcup_{p/q\in \P_B(w_1,\cdots,w_n)}B\Big(p/q, \psi(3^{n+t})\Big).
$$ So we have $$
\limsup_{n\to \infty}E_n\subset \widetilde{\mathcal{W}}\cap B.
$$


The assumption (\ref{f1}) implies the balls in $E_n$ are disjoint, so \begin{align*}
\mu(E_n)&=\sum_{(w_1,\cdots,w_n)\in \F_n}\ \sum_{p/q\in \P_B(w_1,\cdots,w_n)}\mu\Big(B(p/q, \psi(3^{n+t}))\Big)\\
&\ge c_2 2^n\cdot c_1n\cdot c\Big(\psi(3^{n+t})\Big)^{\gamma}.
\end{align*}
Thus, $$
\sum_{n=1}^{\infty}\mu(E_n)\ge \sum_{n=1}^{\infty}c_2 2^n\cdot c_1n\cdot c\Big(\psi(3^{n+t})\Big)^{\gamma}=\infty.
$$

%

To use Chung-Erd\"{o}s inequality, we need estimate the measure of the intersection $E_m\cap E_n$ with $m<n$. Let $A_m$ be a generic ball in $E_m$, so it is of radius $\psi(3^{m+t})$. It is trivial that $$
\mu(E_m\cap E_n)=\sum_{A_m\in E_m}\mu(A_m\cap E_n).
$$
So, we pay attention to the number of balls in $E_n$ which can have non-empty intersection with $A_m$.\smallskip

Recall that by (\ref{f1}), for each ball $A_n\in E_n$, the intersection $A_n\cap \K$ is contained in a cylinder of order $n+t$. On the other hand, each cylinder of order $n+t$ can intersect at most $c_1n\le n$ balls in $E_n$.

 Three cases will be distinguished according to the comparison of the radius of $A_m$ and the gaps between the balls $A_n$ in $E_n$.

(i). $\psi(3^{m+t})\ge 3^{-(n+t)}$. In this case, the ball $A_m$ may intersect many cylinders of order $n+t$, so we count the number of these cylinders.
By a volume argument with respect to the Cantor measure $\mu$, the ball $A_m$ can intersect at most $$\frac{\mu(A_m)}{\mu(I_{n+t})}+2\le 2c^{-1}\Big(3^{n+t}\cdot \psi(3^{m+t})\Big)^{\gamma}
$$ cylinders of order $(n+t)$. So $A_m$ can intersect at most
 $$n\cdot 2c^{-1}\Big(3^{n+t}\cdot \psi(3^{m+t})\Big)^{\gamma}=2c^{-1}\cdot n2^{n+t}\cdot \Big(\psi(3^{m+t})\Big)^{\gamma}$$
  balls $A_n$ in $E_n$.
 Thus \begin{align*}
\mu(E_m\cap E_n)&\le \sum_{A_m\in E_m}\sum_{A_n\in E_n, A_n\cap A_m\ne \emptyset}\mu(A_n)\\
&\le \sum_{A_m\in E_m} 2c^{-1}\cdot n2^{n+t}\Big(\psi(3^{m+t})\Big)^{\gamma}\cdot c^{-1}\Big(\psi(3^{n+t})\Big)^{\gamma}\\
&\le m2^m \cdot 2c^{-2}\cdot  n2^{n+t}\Big(\psi(3^{m+t})\Big)^{\gamma}\cdot \Big(\psi(3^{n+t})\Big)^{\gamma}\\
&=2c^{-2}\cdot m2^m\Big(\psi(3^{m+t})\Big)^{\gamma}\cdot n2^n\Big(\psi(3^{n+t})\Big)^{\gamma}\cdot 2^{t}\\
&\le c_3\cdot \frac{\mu(E_m)\cdot \mu(E_n)}{\mu(B)},
\end{align*} for some absolute constant $c_3>0$. Note that $\mu(B)=2^{-t}$.
\smallskip

(ii). $\frac{1}{2}\cdot 3^{-n-t-k_n}\le \psi(3^{m+t})< 3^{-(n+t)}$.
In this case, the ball $A_m$ can intersect at most one cylinder of order $(n+t)$. By Observation (2), the centers of the balls of $E_n$ intersecting a same cylinder of order $n+t$ are $3^{-n-t-k_n}$-separated which is larger than twice of the radius of a ball in $E_n$ by (\ref{f1}). So for each ball $A_n\in E_n$ which can intersect $A_m$, we thicken it to a ball with the same center and radius $\frac{1}{2}\cdot 3^{-n-t-k_n}$. Then these thicken balls are still disjoint, centered at $\K$ and contained in $4A_m$. Still by a volume argument with respect to the Cantor measure $\mu$, the ball $A_m$ can intersect at most $$ c^{-2}8^{\gamma}\cdot
\Big(\psi(3^{m+t})\cdot 3^{n+t+k_n}\Big)^{\gamma}$$ balls in $E_n$. Thus, by noting that $3^{k_n\gamma}=n$,\begin{align*}
\mu(E_m\cap E_n)
&\le m2^m\cdot c^{-2}8^{\gamma}\cdot
\Big(\psi(3^{m+t})\cdot 3^{n+t+k_n}\Big)^{\gamma}\cdot c^{-1}\Big(\psi(3^{n+t})\Big)^{\gamma}\\
&=c^{-3}8^{\gamma}\cdot m2^m \Big(\psi(3^{m+t})\Big)^{\gamma}\cdot n2^n\Big(\psi(3^{n+t})\Big)^{\gamma}\cdot 2^{t}\\
&\le c_3\cdot \frac{\mu(E_m)\cdot \mu(E_n)}{\mu(B)}.
\end{align*} \smallskip

(iii). $\psi(3^{m+t})< \frac{1}{2}\cdot 3^{-n-t-k_n}$.
In this case, by Observation (1) (2), the ball $A_m$ can intersect at most one ball of $E_n$. Thus\begin{align*}
\mu(E_m\cap E_n)
&\le c^{-1}\cdot m2^m \cdot \Big(\psi(3^{n+t})\Big)^{\gamma}.
\end{align*}

In a summary, we have shown that for $m<n$, $$
\mu(E_m\cap E_n)\le c_3\frac{\mu(E_m)\mu(E_n)}{\mu(B)}+c^{-1}\cdot m2^m \cdot \Big(\psi(3^{n+t})\Big)^{\gamma}.
$$
So, \begin{align*}
\sum_{1\le m\ne n\le N}\mu(E_m\cap E_n)&\le c_3\cdot \frac{1}{\mu(B)}\Big(\sum_{1\le m\le N}\mu(E_m)\Big)^2+2c^{-1}\sum_{n=1}^N\sum_{m=1}^n m2^m \Big(\psi(3^{n+t})\Big)^{\gamma}\\
& \le c_3\cdot \frac{1}{\mu(B)}\Big(\sum_{1\le m\le N}\mu(E_m)\Big)^2+4c^{-1}\sum_{n=1}^N n2^n \Big(\psi(3^{n+t})\Big)^{\gamma}\\
&\le c_3\frac{1}{\mu(B)}\Big(\sum_{1\le m\le N}\mu(E_m)\Big)^2+c_3\sum_{n=1}^N \mu(E_n).
\end{align*}
Therefore, by Chung-Erd\"{o}s inequality, we conclude that $$
\mu(\widetilde{\mathcal{W}}\cap B)\ge \mu(\limsup_{n\to \infty}E_n)\ge c_3^{-1}\mu(B).
$$

Finally any ball $B=B(x,r)$ with $x\in \K$ can contain a cylinder of order $\ell$ and contained in a cylinder of order $\ell-1$. So by the doubling property of $\mu$ (see Lemma \ref{l1}), we conclude that for any ball $B(x,r)$,$$
\mu(\widetilde{\mathcal{W}}\cap B)
\ge 2^{-1}c_3^{-1}\cdot \mu(B).
$$  Thus the full measure property of $\widetilde{\mathcal{W}}$ follows by Lemma \ref{l2}.

\section{More words on Mahler's first question}

Corollary \ref{c1} gives the divergence theory of Mahler's first question, however, as mentioned before, for a rational number $p/q\in \K$, we donot know the greatest common divisor of $p_{\text{int}}$ and $q_{\text{int}}$.
This is the main difficulty for the convergence theory.

In fact, the convergence theory is highly related to the number of {\em reduced} rational numbers in $\K$ with a prescribed range of the denominators. More precisely, define $$
\mathcal{N}_n=\Big\{p/q\in \K: p/q\ {\text{reduced}},\ 3^{n-1}<q\le 3^n\Big\},
$$ and $$
\mathcal{P}_m:=\Big\{p'/q'\in \K: {\text{purely periodic, and reduced}}, 3^{m-1}<q'\le 3^m\Big\}.
$$
The expected goal is to estimate the cardinality of $\mathcal{N}_n$, which has close relations to $\mathcal{P}_m$ (Lemma \ref{l6}). So effective estimations on $\sharp \mathcal{N}_n$ or $\sharp \mathcal{P}_m$ will lead to the convergence theory of Mahler's question.

 Now we give some simple observations on $\mathcal{N}_n$.

Let $p/q\in \K$ with $\mathcal{L}(p/q)=\ell$ and $\mathcal{P}(p/q)=m$. So we write it as $$
p/q=[\ep_1,\cdots,\ep_{\ell}, (\ep_{\ell+1},\cdots,\ep_{\ell+m})^{\infty}].
$$
Then by Definition \ref{d1}, 
$q_{\text{int}}={3^{\ell}(3^m-1)}$ and \begin{align}
p_{\text{int}}&={3^{\ell+m}[\epsilon_1,\cdots,\epsilon_{\ell}, \ep_{\ell+1},\cdots,\ep_{\ell+m}]-3^{\ell}[\epsilon_1,\cdots,\epsilon_{\ell}]}\label{f8}\\
&=(3^m-1)3^{\ell}[\ep_1,\cdots, \ep_{\ell}]+3^m[\ep_{\ell+1},\cdots,\ep_{\ell+m}].\label{f9}
\end{align}

\begin{lem}\label{l5} Use $gcd$ to denote the greatest common divisor. Then
$$gcd(p_{\text{int}},\ 3^{\ell})=1, \ \ \ gcd(p_{\text{int}},\ q_{\text{int}})=gcd(3^m[w_1,\cdots, w_m],\ 3^m-1).$$
\end{lem}\begin{proof}
  If $p_{\text{int}}$ is a multiple of $3$, by (\ref{f8}), one has that $$
\epsilon_{\ell}=\epsilon_{\ell+m}.
$$ As a result, $$
[\epsilon_1,\cdots, \epsilon_{\ell}, (\epsilon_{\ell+1},\cdots, \epsilon_{\ell+m})^{\infty}]=[\epsilon_1,\cdots, \epsilon_{\ell-1}, (\epsilon_{\ell},\cdots, \epsilon_{\ell+m-1})^{\infty}].
$$ This contradicts the definition of $\mathcal{L}(p/q)$. So the first assertion follows. The second assertion follows from the first one and (\ref{f9}).
\end{proof}
\begin{lem}\label{l4} For any $p/q\in \mathcal{N}_n$, one has
  $\mathcal{L}(p/q)\le n$ and $\mathcal{L}(p/q)+\mathcal{P}(p/q)\ge n$.
\end{lem}
\begin{proof} Let $\mathcal{L}(p/q)=\ell$ and $\mathcal{P}(p/q)=m$. Then by Lemma \ref{l5} and since $p/q\in \mathcal{N}_n$, it follows that
$$3^{\ell}\le q\le 3^n,  \ \ \ 3^{n-1}<q<3^{\ell+m}.$$
\vskip -15pt\end{proof}

In fact, Lemma \ref{l5} indicates that to understand the reduced form of $\frac{p_{\text{int}}}{q_{\text{int}}}$,
one need be clear about the reduced form of purely periodic 3-adic expansion. More precisely, we give the following estimation on $\mathcal{N}_n$. Define $$
\mathcal{P}_m:=\Big\{p'/q': {\text{purely periodic reduced}}, 3^{m-1}<q'\le 3^m\Big\}.
$$ \begin{lem}\label{l6} One has\begin{align*}
  \sharp\mathcal{N}_n\le \sum_{m\le n}2^{n-m}\sharp\mathcal{P}_{m}.
\end{align*}
\end{lem}
\begin{proof}
By Lemma \ref{l4}, one has
\begin{align*}
  \sharp\mathcal{N}_n&=\sharp\Big\{p/q: {\text{reduced}}\ 3^{n-1}<q\le 3^n\Big\}\\
&=\sum_{\ell\le n}\sharp\Big\{p/q=[\epsilon_1,\cdots, \epsilon_{\ell}, (w_1,\cdots, w_m)^{\infty}]: \mathcal{L}(p/q)=\ell, \mathcal{P}(p/q)=m, {\text{reduced}}\ 3^{n-1}<q\le 3^n\Big\}.
\end{align*} Let $$p/q=[\epsilon_1,\cdots, \epsilon_{\ell}, (w_1,\cdots, w_m)^{\infty}]$$ with $\mathcal{L}(p/q)=\ell, \mathcal{P}(p/q)=m$
and let $$p'/q'=[(w_1,\cdots, w_m)^{\infty}]=\frac{3^m[w_1,\cdots, w_m]}{3^m-1}$$ be in its reduced form. It is trivial that $$
q'=\frac{3^m-1}{gcd(3^m[w_1,\cdots, w_m], 3^m-1)}, \ q=\frac{(3^m-1)3^{\ell}}{gdc(p_{\text{int}},\ q_{\text{int}})}
$$
By Lemma \ref{l5}, we know that the two denominators are the same, so it follows that $$
q'=q/3^{\ell}.
$$
Thus one has \begin{align*}
  \sharp\mathcal{N}_n&\le \sum_{\ell\le n}2^{\ell}\sharp\Big\{p'/q'=[(w_1,\cdots, w_m)^{\infty}]: {\text{reduced}}\ 3^{n-1-\ell}<q'\le 3^{n-\ell}\Big\}\\
&=\sum_{\ell\le n}2^{\ell}\sharp\mathcal{P}_{n-\ell}=\sum_{m\le n}2^{n-m}\sharp\mathcal{P}_{m}.
\end{align*}
\vskip -15pt
\end{proof}Thus if one can show that $$
\sharp\mathcal{P}_m\ll 2^m,
$$ one would have that $$
\sharp\mathcal{N}_n\le 2^n\sum_{m\le n}\frac{\sharp\mathcal{P}_{m}}{2^{m}}\ll n2^n.
$$

As a corollary of Lemma \ref{l4}, by writing $$
 \mathcal{A}=\Big\{p/q\in \K:  p/q\ {\text{reduced}},\ 3^{n-1}<q\le 3^n,\ \mathcal{P}(p/q)\le M+\log_2 n\Big\},
 $$ for any fixed $M>0$, we have the following.
 \begin{cor}
   $\sharp \mathcal{A}\le 2^{M+2}\cdot n\cdot 2^n.$
 \end{cor}
 \begin{proof} By Lemma \ref{l4},
   \begin{align*}
      \sharp\mathcal{A}&=\sharp\Big\{p/q\in \K:  p/q\ {\text{reduced}},\ 3^{n-1}<q\le 3^n, \  \mathcal{L}(p/q)\le n,\ \mathcal{P}(p/q)\le M+\log_2 n\Big\}\\
      &\le \sharp\Big\{[\epsilon_1,\cdots,\epsilon_{\ell}, (\epsilon_{\ell+1},\cdots, \epsilon_{\ell+m})^{\infty}]:  \ell\le n, \ m\le M+\log_2 n, \  \ep_i\in \{0,2\}, 1\le i\le \ell+m\Big\}\\
      &= \sum_{\ell\le n}\ \sum_{m\le M+\log_2 n}2^{\ell+m}\\
      &\le 2^{M+2}\cdot n\cdot 2^n.
   \end{align*}
\vskip-15pt \end{proof}

 As a consequence, by writing $$
   \mathcal{W}^{(1)}(\psi)=\Big\{x\in \K: |x-p/q|<\psi(q), \ {\text{i.m.}}\ p/q\ {\text{reduced}}, p/q\in \K, \mathcal{P}(p/q)\le \log_2 \log q \Big\}
   $$we have the following,\begin{cor}$$
  \sum_{n\ge 1}n\cdot 2^n \Big(\psi(3^n)\Big)^{\gamma}<\infty \Longrightarrow \mu\Big(\mathcal{W}^{(1)}(\psi)\Big)=0.$$
 \end{cor}

 Thus, the only problem left is to consider the set $$
   \mathcal{W}^{(2)}(\psi)=\Big\{x\in \K: |x-p/q|<\psi(q), \ {\text{i.m.}}\ p/q\in \K\ {\text{reduced}}, \mathcal{P}(p/q)> \log_2 \log q \Big\}
   $$

\begin{rem}
 It is noticed that the quantity $\log_2 n$ occurs at two places: one is in the divergence part to seek well spaced rational numbers in $\K$ (see the remark in the end of Section 4); the other is in the convergence part for the period of a rational number (see $\mathcal{A}$). We donot know whether this is a coincidence or some deep relations there.

For the cardinality of $\mathcal{N}_n$, there are some progress recently in \cite{Rahm, Sch}. But it is still far from being clear. The following are some conjectures:\begin{itemize}
  \item in \cite{Broderich3} by Broderick, Fishman \& Reich $$
\sharp \mathcal{N}_n\ll 2^{\tau n}, \ {\text{for some }}\ \tau<2.
$$

\item in \cite{FishS} by Fishman \& Simmons $$
\sharp \mathcal{N}_n\ll 2^{(1+\epsilon)n}, \ {\text{for any }}\ \epsilon>0.
$$

\item We also pose an ambitious conjecture $$
\sharp \mathcal{P}_m\ll 2^m, \ \ {\text{or}}\ \ \sharp \mathcal{N}_n\ll n\cdot 2^{n}.
$$
\end{itemize}
So, if the last conjecture is true, a complete metric theory for Mahler's first question could be established: {\em almost all or almost no points in $\K$ are intrinsically $\psi$-well approximable according to $$
\sum_{n\ge 1}n\Big(3^n\psi(3^n)\Big)^{\gamma}=\infty\ {\text{or}}\ <\infty.
$$}
\end{rem}

\section*{Acknowledgement} The authors show their sincerely appreciations to Prof. Teturo Kamae (Osaka City University) for his generosity in sharing the proof of Lemma \ref{l3}.

\end{document}